\documentclass[12pt,a4paper]{article}
\usepackage[english]{babel}
\usepackage{graphicx} 
\usepackage{latexsym}
\usepackage{amssymb}
\usepackage{graphicx,url}
\usepackage{amsmath,graphics,epsfig,url}
\usepackage{amsfonts}
\usepackage{hyperref}
\usepackage{color}
\usepackage{amsthm}
\usepackage{tikz}
\usepackage{latexsym}
\usetikzlibrary{arrows}

\newcommand{\executeiffilenewer}[3]{%
\ifnum\pdfstrcmp{\pdffilemoddate{#1}}%
{\pdffilemoddate{#2}}>0%
{\immediate\write18{#3}}\fi%
}

\theoremstyle{plain}
\newtheorem{theo}{Theorem}[section]
\newtheorem{proposition}[theo]{Proposition}
\newtheorem{lemma}[theo]{Lemma}
\newtheorem{defi}[theo]{Definition}

\newtheorem{corollary}[theo]{Corollary}

\input amssym.def
\newsymbol\rtimes 226F
\newfont{\nset}{msbm10}
\newcommand{\ns}[1]{\mbox{\nset #1}}

\def\Z{\ns Z}

\def\Z{\ns{Z}}

\def\dist{\mathop{\rm dist}\nolimits}
\def\ecc{\mathop{\rm ecc}\nolimits}

\def\swap{\mathop{\rm swap}\nolimits}
\def\>{\mathop{\rightarrow}\nolimits}

\def\n{\mbox{\bf n}}
\def\sn{\mathbf{n}}

\def\r{\mbox{\bf r}}
\def\u{{\mbox {\bf u}}}
\def\v{\mbox{\bf v}}
\def\w{{\mbox {\bf w}}}
\def\x{\mbox{\bf x}}
\def\y{\mbox{\bf y}}
\def\0{\mbox{\bf 0}}

\def\vecalpha{\mbox{\boldmath $\alpha$}}

\def\alt{\mathop{\rm alt}\nolimits}

\title{A new general family of deterministic\\ hierarchical networks
\thanks{Corresponding author: M.A. Fiol, Dept. Matem\`{a}tica Aplicada IV,
Universitat Polit\`{e}cnica de Catalunya, Jordi Girona 1-3,
08034 Barcelona (Catalonia), Tel: +34 93
4015993, Fax: +34 93 4015981, e-mail: {\tt fiol@ma4.upc.edu}} }

\author{C. Dalf\'o$^a$, M.A. Fiol$^{a,b}$\\
{\small $^a$Departament de Matem\`atica Aplicada IV }\\
{\small Universitat Polit\`ecnica de Catalunya }\\
{\small Barcelona (Catalonia) }\\
{\small $^{b}$Barcelona Graduate School of Mathematics} \\
{\small E-mails: {\tt \{cdalfo,fiol\}@ma4.upc.edu}}}

\begin{document}
\maketitle

\begin{abstract}
It is known that many networks modeling real-life complex systems are small-word (large local clustering and small diameter) and scale-free (power law of the degree distribution), and very often they are also hierarchical.
Although most of the models are based on stochastic methods, some deterministic constructions have been recently proposed, because this allows a better computation of their properties.
Here a new deterministic family of hierarchical networks is presented, which generalizes most of the previous proposals, such as the so-called binomial tree. The obtained graphs can be seen as graphs
on alphabets (where vertices are labeled with words of a given alphabet, and the edges are defined by a specific rule relating different words). This allows us the characterization of their main distance-related parameters, such as the radius and the diameter. Moreover, as a by product, an efficient shortest-path local algorithm is proposed.
\end{abstract}

%

\section{Introduction}
\subsection{Models for complex networks}
Since the papers of Watts and
Strogatz~\cite{ws98} on small-world networks and by Barab\'asi and
Albert~\cite{ba99} on scale-free networks, there has been a
special interest in the theoretical study of complex networks, such as the World Wide Web~\cite{ajb99}, some kind of social networks \cite{rb03,n03,rsmob02}, communication networks \cite{tp12}, protein networks \cite{JeMaBaOl01,wrb03}, etc.
Two main characteristic of such networks are the presence of a strong
local clustering (that is, groups of nodes with mutual interconnections), high modularity, certain distribution of the degrees,
and self-similarity.

In this paper we propose a new deterministic family, which generalizes some previous proposals of hierarchical networks \cite{rsmob02,rb03,n03,bccf09,bcdf14}. As a first approach, our family of hierarchical networks is defined
recursively from an initial complete graph on $n_k$ vertices.  Also, it is shown that the obtained graphs can also be seen as graphs on alphabets \cite{gfy92}. The vertices of such graphs are labeled with words of a given alphabet, and their edges
are defined by some specific rules relating different words. This allows us the
characterization of their main distance-related parameters, such as the vertex eccentricities, the radius
and the diameter. Moreover, as a by product, an efficient shortest-path local
algorithm is proposed.

\subsection{Some basic notation}
Here we present the basic notation used throughout the paper.
Let $G=(V,E)$ be a finite, simple, and connected graph
with vertex set $V$, order $N=|V|$, edge set $E$, and size $M=|E|$.
If the vertices $u$ and $v$ are adjacent, $uv\in E$, we represent it by $u \sim v$. The {\em distance} between vertices $u$ and $v$ is denoted by $\dist(u,v)$.
The {\em eccentricity} of a vertex $u$ is $\varepsilon (u)= \min\{\dist(u,v): v\in V \}$. Hence, the {\em radius} and {\em diameter} of $G$ are, respectively,
$R=\min_{u\in V}\ecc(u)$, and $D=\max_{u\in V}\ecc(u)$. The set
of vertices at distance $i$ from a given vertex $u\in V$ is $G_i(u)=\{v:\dist(u,v)=i\}$, for $i=0,\ldots,D$. Then, the {\em degree} of a vertex $u$ is just $\delta(u)=|G_1(u)|$.

\section{The hierarchical graph  $H_{\sn}^k$}
In this section we generalize the constructions of deterministic
hierarchical graphs introduced by Ravasz and Barab\'asi~\cite{rb03}, Ravasz, Somera, Mongru, Oltvai, and Barab\'asi~\cite{rsmob02}, Noh~\cite{n03}, and
Barri\`ere, Comellas, Dalf\'o and Fiol~\cite{bcdf14}. Roughly
speaking, our graphs are constructed in the following way: Given $k$ integers $n_1,n_2,\ldots,n_k$, we first  connect a
selected root vertex of the complete graph $K_{n_1}$ to some vertices of
$n_{2}-1$ copies of $K_{n_1}$, and then we add some new edges between
such copies. This gives a graph with $n_{1}n_2$ vertices. Next,
$n_{3}-1$ replicas of the new whole structure are added, again with some
edges between them and to the same root vertex. At this step the
graph has $n_{1}n_{2}n_3$ vertices. Then we iterate the process until the
desired graph, with order $N=n_1n_2\cdots n_k$, 
is obtained.
We give below two formal definitions.


\subsection{A recursive definition}
A recursive definition of the considered networks is as follows.

\begin{figure}[t]
\begin{center}
\vskip-0.75cm
\includegraphics[width=12cm]{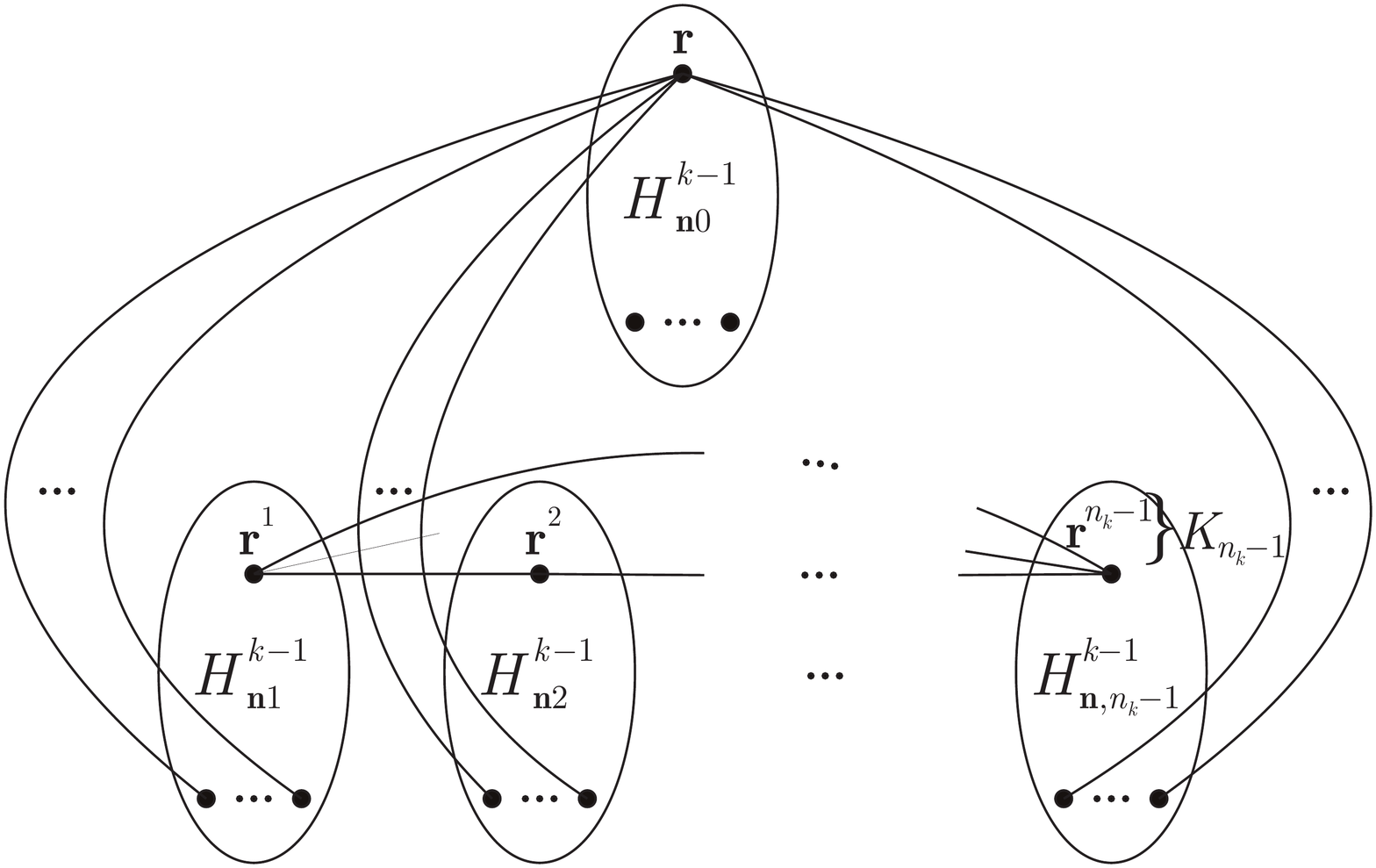}
\vskip-1cm
\caption{The graph $H_{\sn}^k$ from $n_k$ copies of $H_{\sn}^{k-1}$. }
\label{H(n,k)}
\end{center}
\end{figure}

\begin{defi}
\label{recur-defi}
Let $n_1,n_2,\ldots$  a sequence of positive integers, $n_i\geq 2$, whose finite $k$-subsequences $n_1,\ldots,n_k$ are abbreviated by the symbols $_{\sn}^k$, $k=1,2,\ldots$.   The {\em hierarchical graph}
$H_{n_1,\ldots,n_k}$, also denoted by $H_{\sn}^k$, has vertex set $V_{\sn}^k$, with $N=n_1n_2\cdots n_k$ vertices, each
denoted by a $k$-string $x_1 x_2 x_3\ldots x_{k}$, where $x_i \in
\Z_{n_i}$ for $i=1,\ldots,k$, and edge set $E_{\sn}^k$ defined recursively
as follows:
\begin{itemize}
\item
$H_{\sn}^1\equiv H_{n_1}$ is the complete graph $K_{n_1}$.
\item
For  $k>1$,  $H_{\sn}^k$ is obtained from the union of $n_k$ copies of
$H_{\sn}^{k-1}$, each denoted by $H_{\sn\alpha}^{k-1}$,
and with vertices $x_1 x_2\ldots x_{k_1}\alpha$, for every fixed $\alpha=0,\ldots, n_k-1$, by adding the following new edges:
\begin{eqnarray}
\label{adj2'}
00\ldots 00  & \ \sim \ &  x_1 x_2 \ldots x_{k-1}
x_{k}, \qquad    \mbox{$x_i\neq 0$, for $i=1,\ldots,k$}; \\
\label{adj3'}
00\ldots 0 x_{1} & \ \sim  \ &  00\ldots 0y_{1},
\hskip 1.9cm \mbox{$x_{1},y_{1}\neq 0$, and $\ x_{1} \neq y_{1}$.}
\end{eqnarray}
\end{itemize}
\end{defi}

To illustrate the recursive procedure, Fig.~\ref{H(n,k)} shows the
hierarchical graph $H_{\sn}^k$ obtained by joining $n_k$ copies of $H_{\sn}^{k-1}$.
Notice that
vertex $\r:=00\ldots 0$, which we distinguish and call the \emph{root},
is adjacent by (\ref{adj2'}) to vertices $x_1 x_2\ldots x_{k}$, for
$x_i\neq 0$ and $i=1,\ldots,k$, which we call
\emph{peripheral vertices}.

In particular, if all the numbers $n_i$, $i=1,\ldots,k$, are equal to, say $n$, we will denote the corresponding graph as $H_{n^k}$.

\subsection{A definition as a graph on alphabet}
To give a more direct definition of $H_{\sn}^k$, it is convenient to introduce the following notation.
The length $\ell$ (number of elements) of a given string $\x=x_1\ldots x_{\ell}$ will be denoted by $|\x|$.  If all its elements are nonzero, we represent it by $\x^*=x_1^*\ldots x_{\ell}^*$. A string
with all its elements $0$ is denoted by $\0$.
Then, the edge set $E_{\sn}^k$ is characterized by the following adjacency rules (the substrings of each vertex have appropriate lengths and, hence, they sum up to $k$):
\begin{eqnarray}
x_1\x& \ \sim \ &  y_1\x,
\qquad \mbox{where $x_1\neq y_1$} \label{adj0}; \\
\0\x      & \ \sim \ &  \y^*\x,
\qquad \mbox{where $|\y^*|=|\0|$} \label{adj1}; \\
\0x_i^*\x  & \ \sim  \ &  \0y_i^*\x ,
\qquad \mbox{where $y_i^*\neq x_i^*$.} \label{adj2}
\end{eqnarray}

As a concrete example, Fig.~\ref{H_{2,3,4}} shows two different drawings of the graph $H_{2,3,4}$.
\begin{figure}[t]
\begin{center}
\vskip-0.75cm
\includegraphics[width=14cm]{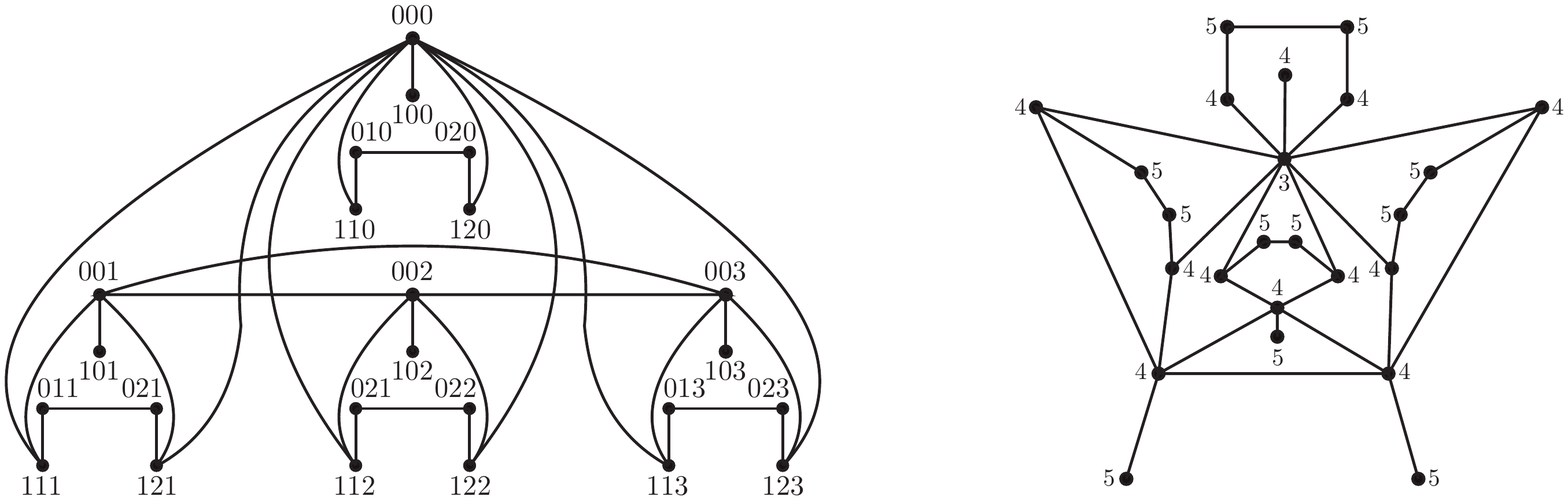}
\vskip-5.5cm
\caption{Two representations of the graph $H_{2,3,4}$: On the left a non-planar drawing with the labels of its vertices, and on the right a planar drawing with the eccentricities of its vertices.}
\label{H_{2,3,4}}
\end{center}
\end{figure}

\subsection{Some particular cases}

\subsubsection{The hierarchical network $H_{n,k}$}

The hierarchical network introduced by Barri\`ere, Comellas, and the authors
(see \cite{bcdf14}), which was denoted by $H_{n,k}$ is just $H_{n^k}$ (that is, $n_1=\cdots=n_k=n$).
In turn, $H_{n^k}$ generalizes the deterministic hierarchical network
introduced by Ravasz, Somera, Mongru, Oltvai, and Barab\'asi~\cite{rsmob02}, see also Barab\'asi and Oltvai,~\cite{bo04} (which corresponds to $H_{4^k}$). Moreover, the deterministic hierarchical networks introduced by Ravasz and Barab\'asi~\cite{rb03} and generalized
by Noh~\cite{n03}, constitute  a subgraph of $H_{5^k}$ (some edges are not present).

\subsubsection{The binomial tree $B_k$}
As it is well-known, a binomial tree is used in Computer Science to  model a recursive data structure. A tree of degree zero $B_0$ is a singleton, in our context, allowing sequences of length zero, we could denote it as  $B_0=H_{2^0}$. A tree of degree $k$, $B_k$, is constructed from two trees of degree $k-1$, by joining their two roots. Then, within our notation, $B_k=H_{2^k}$.
Figure \ref{binomial trees} shows the binomial trees of order $k=0,\ldots,4$.
Notice that $B_k$ can also recursively defined by saying that its root node has as children the roots of binomial trees of orders $k-1,k-2,...,2,1,0$.
\begin{figure}[!h]
\begin{center}
\includegraphics[width=14cm]{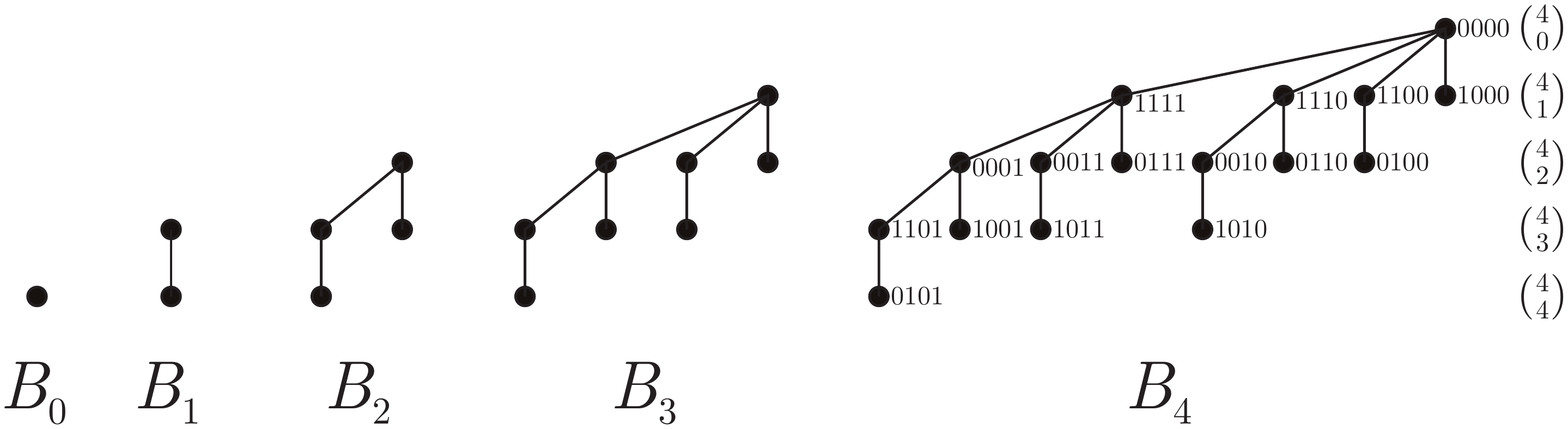}
\vskip-6cm
\caption{The first binomial trees $B_i$, $i=0,\ldots,4$.}
\label{binomial trees}
\end{center}
\end{figure}

The binomial tree  can also be seen as the hierarchical product  of $k$ copies of the complete grapk on 2 vertices $K_2$,  denoted
$B_k=K_2\sqcap\stackrel{k}{\cdots}\sqcap K_2$ (see Barri\`ere, Comellas Dalf\'o, and Fiol \cite{bcdf09}).

In due course, we will show that the number of vertices $\x$ of $H_{2^k}$ at distance $i$ from the root $\r=\0$  is the binomial coefficient ${k\choose i}$ (see Fig.~\ref{binomial trees}). Of course, this property is the reason for the name of such a structure.

\section{Hierarchical properties}
\label{hierarquicalprop}

The main structural  properties of the graphs $H_{\sn}^k$ are the following:

\subsection{Order and size}
We have already seen that the order of $H_{\sn}^k$
is $N=|V_{\sn}^k|=\prod_{i=1}^k n_i$. With respect to its size, we have the following result:
\begin{lemma}
Let $\n\equiv n_1,n_2,\ldots, n_k$. Then, the number of edges of $H_{\sn}^k$ is
\begin{equation}\label{recur-size}
M=|E_{\sn}^k|={n_1\choose 2}\prod_{i=2}^k n_i
+\sum_{i=2}^k\prod_{j=1}^i (n_j-1)\prod_{j=i+1}^k n_j
+\sum_{i=2}^k {n_i-1\choose 2}\prod_{j=i+1}^k n_j.
\end{equation}
\end{lemma}
\begin{proof}
Let $M_h$ be the size of $H_{\sn}^h$. Then, by using the recursive Definition \ref{recur-defi},
we have $M_1=n_1$ and
$$
M_h=n_h M_{h-1}+\prod_{i=1}^h (n_h-1)+{n_h-1\choose 2}.
$$
Note that the first summand corresponds to the number of edges of the $n_h$ copies of $H_{\sn}^{h-1}$, whereas the second and third summands account for the edges  joining such copies, according to the conditions \eqref{adj2'} and \eqref{adj3'}, respectively. Using this, the computation of \eqref{recur-size} is immediate.
\end{proof}

Of course, the 1st, 2nd, and 3rd terms in \eqref{recur-size} correspond also to the adjacency conditions \eqref{adj0}, \eqref{adj1}, and \eqref{adj2}, respectively. For instance, in the graph $H_{2,3,4}$ of Fig.~\ref{H_{2,3,4}}, we get $M=12+18+3=33$.

\subsection{Subgraphs}
The following lemma, whose proof follows easily by seeing  $H_{\sn}^k$ as a graph on alphabet, shows the hierarchical nature of our networks.

\begin{lemma}
\label{lema-subgraphs}
\begin{itemize}
\item[$(i)$]
For each given sequence $\alpha_2\ldots \alpha_k$, with
$\alpha_i\in \Z_{n_i}$, $i=2,\ldots,k$, the vertex set $\{x_1\alpha_1
\alpha_2\ldots \alpha_{k}: x_1\in\Z_{n_1}\}$ induces a subgraph isomorphic to $K_{n_1}$.
\item[$(ii)$]
Let $\n\equiv n_1,\ldots,n_k$. For every $i$, $1\le i \le k-1$, $H_{\sn}^k$ can be decomposed into
$n^i$ vertex-disjoint subgraphs isomorphic to $H_{x_1,\ldots,x_{k-i}}$. Each of
such subgraphs is denoted by $H_{\sn\mbox{\scriptsize
$\vecalpha$}}^{k-i}$, and has vertex labels  $
x_{1}x_{2}\ldots x_{k-i}\vecalpha$, with $\vecalpha\equiv\alpha_{k-i+1},\ldots,
\alpha_k \in \Z_{k-i+1}\times\cdots\times \Z_{k}$ being a fixed sequence.
\item[$(iii)$]
The root vertex of the subgraph
$H_{\sn\mbox{\scriptsize
$\vecalpha$}}^{k-i}$ is $\0\vecalpha $, where $|\0|=k-i$, whereas
its {\em peripheral vertices}  are of the form
$x_{1}^*x_{2}^*\ldots x_{k-i}^*\vecalpha $.
\item[$(iv)$]
By collapsing in $H_{\sn}^k$ each of the  $n_{k-i+1}\cdots n_k$  subgraphs
$H_{\sn\mbox{\scriptsize $\vecalpha$}}^{i}$, with a fixed $\vecalpha\in\Z_{k-i+1}\times\cdots\times\Z_{k}$,
into a single vertex and all multiple edges into one, we obtain a
graph isomorphic to $H_{n_{k-i+1},\ldots,n_k}$.
\medskip
\item[$(v)$]
For every fixed $i$, $1\le i\le k$, and
a given sequence $\vecalpha\in\Z_{i+1}\times\cdots\times \Z_k$, the
 $n_i-1$ vertices labeled $\0 x_i^*\vecalpha$ with
$x_i^* \in \Z_{n_i}^*$ (that is, the root vertices of
$H_{\sn x_i^*\mbox{\scriptsize $\vecalpha$}}^{i-1}$)
induce a complete graph isomorphic to $K_{n_i-1}$.
\end{itemize}
\end{lemma}

\section{Distance parameters}

First, we introduce some notation concerning $H_{\sn}^k$. Let
$\dist_k(\x,\y)$ denote the distance between vertices $\x,\y\in
V_{\sn}^k$ in $H_{\sn}^k$; and $\dist_k(\x,U):=\min_{\u\in
U}\{\dist_k(\x,\u)\}$. Let $\r^{\alpha}=00\ldots0$ be the root
vertex of $H_{\sn\alpha}^{k-1}$, $\alpha\in \Z_n$ (as stated before,
$\r$ stands for the root vertex of $H_{\sn}^k$). Let $P$ and
$P^\alpha$ for $\alpha\in \Z_{n_k}$, denote the set of peripheral vertices
of $H_{\sn}^k$ and $H_{\sn\alpha}^{k-1}$, respectively.

\subsection{The distance function}
In \cite{bcdf14} Barri\`ere, Comellas and the authors proved that the radius $R_k$ of $H_{n,k}$, the
eccentricity $\ecc_k(\r)$ of its root $\r$, and its diameter $D_k$ are
$R_k=\ecc(\r)=k$, and $D_k=2k-1$.
Here we will be more precise and will give both a formula for the distance between vertices, and a proof that the routing algorithm of the next section always gives the shortest path.

Given a $k$-sequence $\x=x_1\ldots x_k\in V_{\sn}^k$, consider the `expanded' $(k+1)$-sequence $\x^+=\x0$ (that is, with $x_{k+1}=0$).
Then, the {\em alternating number} of $\x$, denoted by $\alt(\x)$, is the number of changes in $\x^+$ from a zero element, say $x_i$, to a nonzero element $x_{i+1}^*$ (or vice versa). For instance, if $\x=\x_1^* \0_2 \x_3^* \0_4 \x_5^*$, then $\alt(\x)=5$ (here the $\0_i$'s denote zero strings of some length), and if $\x=\x_1^* \0_2 \x_3^* \0_4$, then $\alt(\x)=3$.
Moreover, given two $k$-sequences $\x$, $\y$,  we denote its maximum common suffix as $\x\cap\y$,
with length $\ell=|\x\cap\y|$, for $0\le \ell\le k$.

By looking at the structure of adjacent vertices given by the adjacency conditions \eqref{adj0}-\eqref{adj2}, the following result is clear.

\begin{lemma}
Let $\x=x_1\ldots\u$ and $\y=y_1\ldots\u$ be two adjacent vertices of  $H_{\sn}^k$ (with $\u$ possibly being the empty string). Then,
\begin{itemize}
\item
If condition \eqref{adj0} holds, then, either $\alt(\x)=\alt(\y)\pm 1$ if $x_1=0$ or $y_1=0$,  or $\alt(\x)=\alt(\y)$ otherwise ($x_1,y_1\neq 0$).
\item
If condition \eqref{adj1} holds, then, $\alt(\x)=\alt(\y)\pm 1$.
\item
If condition \eqref{adj2} holds, then $\alt(\x)=\alt(\y)$.
\end{itemize}
\end{lemma}

\begin{corollary}
\label{coro1}
Let $\x$ be a vertex of  $H_{\sn}^k$ with $\alt(\x)\neq 0$. Then, there exists a vertex $\y \sim \x$ such that $\alt(\y)=\alt(\x)-1$, but no vertex $\y' \sim \x$ satisfies $\alt(\y')<\alt(\x)-1$.
\end{corollary}

\begin{corollary}
\label{coro2}
Let $\x=x_1\ldots x_k$ be a vertex of $H_{\sn}^k$, with root $\r=\0$.
 Then,
\begin{equation}
\label{dist(x,r)}
\dist(\x,\r)=\alt(\x,\r).
\end{equation}
\end{corollary}
\begin{proof}
By Corollary \ref{coro1}, it is clear that, if $x_k=0$ there is a shortest path from $\x$ to $\r=\0$ of length  $\alt(\x)$. Otherwise, if $x_k=x_k^*\neq 0$, there is a shortest path from $\x$ to $\x^*=x_1^*\ldots x_k^*$ of length $\alt(\x)-1$, but $\x^*$ is adjacent to $\0$.
\end{proof}

As an example, note that in the binomial tree $B_k=H_{2^k}$, the vertices at distance $i$ from the root $\r=\0$ are the binary sequences $\x$ satisfying
$\alt(\x)=i$. Thus, its number is the binomial coefficient ${k\choose i}$,
as was commented before.

The following result gives the distance between two generic vertices $\x$, $\y$ of $H_{\sn}^k$. To avoid the trivial case of the complete graph, we will suppose that $k\ge 2$. Moreover, without loss of generality, we can assume that
$\x$ and $\y$ have no common suffix (so that $x_k\neq
y_k$). Otherwise, if $\x=x_{1}x_{2}\ldots x_{k-i}\vecalpha $ and
$\y=y_{1}y_{2}\ldots y_{k-i}\vecalpha $, that is, $\vecalpha=\x\cap
\y$ and $i=|\vecalpha|>0$, we are in the subgraph
$H_{\sn\mbox{\scriptsize $\vecalpha$}}^{k-i}$ (by Lemma \ref{lema-subgraphs}$(iii)$) and, hence, we can apply the
routing algorithm to the vertices $\x'=x_{1}x_{2}\ldots x_{k-i}$ and
$\y'=y_{1}y_{2}\ldots y_{k-i}$ of $H_{\sn}^{k-i}$.

\begin{proposition}
\label{propo-dist}
Let $\x=x_1\ldots \u$ and $\y=y_1\ldots \v$ be two vertices of $H_{\sn}^k$, with no common suffix, and where all elements of $\u$  are either zero or nonzero, and all elements of $\v$ are either nonzero or zero, respectively. Let $\x'$ be the sequence obtained from $\x$ by deleting its rightmost element $x_k$, and similarly for $\y'=\y\setminus y_k$.
Then,
\begin{itemize}
\item[$(i)$]
If either $x_k=0$ or $y_k=0$, then
$$
\dist(\x,\y)= \alt(\x)+\alt(\y).
$$
\item[$(ii)$]
If $x_k,y_k\neq 0$, and  $|\u|,|\v|>1$, then
$$
\dist(\x,\y)= \alt(\x)+\alt(\y).
$$
\item[$(iii)$]
If $x_k,y_k\neq 0$, and either $|\u|=1$ or  $|\v|=1$, then
$$
\dist(\x,\y)= \alt(\x')+\alt(\y')+1.
$$
\end{itemize}
\end{proposition}
\begin{proof}
First, notice that, as $|\x\cap\y|=0$, we always have $x_k\neq y_k$.
Then the key idea is that we cannot reach $\y$ from $\x$ without going through vertices with most of their elements being null (roots).
Thus, using Corollary \ref{coro2}, we have:

$(i)$-$(ii)$ In these cases, the shortest path must go through the root $\r=\0$ and, hence,
$$
\dist(\x,\y)=\dist(\x,\r)+\dist(\r,\y)=\alt(\x)+\alt(\y).
$$

$(iii)$ Now, the shortest path must go through the roots $\r'=\0x_k^*$ and $\r''=\0y_k^*$ of the subgraphs isomorphic to $H_{x_1\ldots,x_{k-1}}$
and $H_{y_1\ldots,y_{k-1}}$, respectively. Then
$$
\dist(\x,\y)=\dist(\x,\r')+\dist(\r',\r'')+\dist(\r'',\y) =\alt(\x')+1+\alt(\y'),
$$
since $\r'$ and $\r''$ are adjacent by \eqref{adj2}.
This completes the proof.
\end{proof}

\subsection{Eccentricity, radius and diameter}

As a consequence of the above results, we have the following lemma:

\begin{lemma}
Let $\x$ be a vertex of $H_{\sn}^k$, as in Proposition \ref{propo-dist}. Then, its
eccentricity $\ecc_k(\x)$,
the radius $R_k$ of $H_{\sn}^k$, and its diameter $D_k$ are:
\begin{enumerate}
\item[$(i)$] For the eccentricity,  we must distinguish two cases:
\begin{itemize}
\item[$(i.1)$] If $\n=2,2,\ldots,2$ (that is, $H_{\sn}^k$ is the binomial tree $B_k$), let $\overline{\x}$ denote the sequence obtained from $\x$ by interchanging $0$'s and $1$'s. Then,
$$
\ecc_k(\x)=
\left\{
\begin{array}{ll}
\alt(\x)+k, & \mbox{if $x_k=0$,}\\
\alt(\overline{\x})+ k, & \mbox{if $x_k=1$.}
\end{array}
\right.
$$
\item[$(i.2)$]
If $H_{\sn}^k$ is not the binomial tree $B_k$, then
$$
\ecc_k(\x)=
\left\{
\begin{array}{ll}
\alt(\x)+k-1, & \mbox{if $|\u|=1$ and $x_k\neq 0$,}\\
\alt(\x)+k,    & \mbox{otherwise.}
\end{array}
\right.
$$
\end{itemize}
\item[$(ii)$]
The radius of $H_{\sn}^k$ coincides with the eccentricity of its root:
$$
R_k=\ecc(\r)=k.
$$
\item[$(iii)$] The diameter of $H_{\sn}^k$ is the same as that of $H_{n^k}$:
$$
D_k=2k-1.
$$
\end{enumerate}
\end{lemma}
\begin{proof}
$(i)$ If $x_k=0$, by Proposition \ref{propo-dist}$(i)$, $\dist(\x,\y)=\alt(\x)+\alt(\y)$, but $\max\{\alt(\y)\}=k$.
This proves que first case in $(i.1)$, and the second one in $(i.2)$ (since the same is true when $|\u|>1$, by Proposition \ref{propo-dist}$(ii)$). The second equality in $(i.1)$ follows from the fact that the mapping $\x \mapsto \overline{\x}$ is an isomorphism on the binomial tree.
Finally, the first equality in $(i.2)$ is proved similarly by using
Proposition \ref{propo-dist}$(iii)$.

$(ii)$ Take $\x=\r$ in $(i)$.

$(iii)$ Take $\x$ with $\alt(x)=k$ in $(i)$.
\end{proof}
\section{A shortest path routing algorithm}

The reasonings that support Proposition \ref{propo-dist}, leads us to a routing algorithm between two generic vertices $\x$, $\y$ of $H_{\sn}^k$, which follows a shortest path.

\subsection{Routing}

Let us consider two vertices in $H_{\sn}^k$, say $\x=x_1x_2\ldots x_k$
and $\y=y_1y_2\ldots y_k$. In Table \ref{pseudocode}, we show a possible
version of the routing algorithm.
As commented  above, the
key idea is to go through, either,  the root $\r=\0$ of $H_{\sn}^k$, or the roots $\r'=\0{k-1}x_k$ and $\r''=\0{k-1}y_k$ of the respective subgraphs.
To describe the algorithm it is useful to introduce some further notation.
Given a sequence $\x=x_1x_2\ldots$, we denote by $\overline{\x}=\overline{x}_1\overline{x}_2\ldots$ the sequence obtained from $\x$ by changing each zero term $x_i$ by an (arbitrary) nonzero term
$\overline{x}_i\in \Z_{n_i}$ and vice versa. Of course, if $n_i=2$, $\overline{x}_i$ is just the  already defined (and univocally determined) conjugate of $x_i$. Moreover, we say that $\x$ is {\em uniform} if all its terms are either zero or nonzero.
If $\x$ has maximum uniform prefix $\u$, we denote by $\swap(\x)$ the sequence obtained from $\x$ by changing $\u$ to $\overline{\u}$.
Similarly, given two sequences of the form $\x=\u_1\u_2\w$, and $\y=\v_1\v_2\w$, where both $\u=\u_1\u_2$ and $\v_2$ are maximal uniform subsequences, one of them being $\0$, and $|\u_2|=|\v_2|$, we denote by $\swap(\x\rightarrow\y$) the sequence $\u_1\v_2\w$. (For a better understanding, see the examples in the next subsection.)

\begin{table}[t]
\begin{center}
{\tt
\begin{tabbing}
{\bf algo}\={\bf rithm:} Routing \\
           \> {\bf input:} \= Sequence $n_1,\ldots,n_k$, \\
           \>              \> Vertices \=$\x=x_1\ldots \u$, $\y=y_1\ldots \v$ (with $\u$ and $\v$ maximal \\
           \>              \>      \>being uniform subsequences)\\
           \> {\bf output:}\> A shortest path between $\x$ and $\y$ \\
           \>{\bf If} $x_k=0$ \= {\bf or} $y_k=0$ {\bf or} $|\u|,|\v|>1$ {\bf then}\\
           \>                 \>{\bf while} \= $\alt(\x)>0$ {\bf do} \\
           \>            \>   \>{\bf go to} $\x:=\swap(\x)$\\
           \>            \>$\r:=\0$\\
           \>            \>{\bf while} $\r\neq \y$ {\bf do}\\
\>           \>          \>{\bf go to} $\r:=\swap(\r \rightarrow \y)$\\
\> {\bf Else} \\
\>$\x':=\x\setminus x_k$, $\y':=\x\setminus y_k$,\\
\>        \>{\bf while} $\alt(\x')>0$ {\bf do}\\
\>        \>           \>{\bf go to} $\x':=\swap(\x')$\\
\>        \>$\x:=\x'\cup x_k$\\
\>        \>$\r:=\0y_k$\\
\>        \> {\bf while} $\r\neq \y$ {\bf do}\\
\>        \>            \>{\bf go to} $\r:=\swap(\r \rightarrow \y)$\\
{\bf end}
\end{tabbing}
}
\caption{A shortest path routing algorithm.}
\label{pseudocode}
\end{center}
\end{table}

\subsection{Examples}
Let us now see three examples of the outputs of our routing algorithm.
Each example corresponds to one of the cases of Proposition \ref{propo-dist}.
According to this result, in each case we have taken
vertices $\x$ and $\y$ with no common suffix, so that $x_k\neq
y_k$.

\begin{itemize}
\item
If $\x=01010$, and $\y=10101$ in $H_{2^5}$, we have $\alt(\x)=4$ and $\alt(\y)=5$. Since $x_5=0$, by Proposition \ref{propo-dist}$(i)$, $\dist(\x,\y)=4+5=9$. Indeed, the shortest path is:
\item[]
$01010\rightarrow 11010 \rightarrow 00010 \rightarrow 11110 \rightarrow 00000 \rightarrow 11111 \rightarrow 00001 \rightarrow 11101 \rightarrow 00101\rightarrow 10101$.

\item
If $\x=0210312$, and $\y=1221023$ in $H_{2,3,3,3,4,4,4}$, we have $\alt(\x)=4$ and $\alt(\y)=3$. Since $|\u|,|\v|>1$, Proposition \ref{propo-dist}$(ii)$ tells us again that $\dist(\x,\y)=4+3=7$. In this case, one of the possible shortest paths is:
\item[]
$0210312\rightarrow 1210312 \rightarrow 0000312 \rightarrow 1111112 \rightarrow 0000000 \rightarrow 1111123 \rightarrow 0000023 \rightarrow 1221023$.

\item Let $\x=102302$, and $\y=101013$ in $H_{2,2,3,4,2,4}$. Since $x_6,y_6\neq 0$ and $|\u|=1$, we must apply Proposition \ref{propo-dist}$(iii)$ with $\alt(\x')=3$ and $\alt(\y')=5$, which gives   $\dist(\x,\y)=3+5+1=9$. In this case, one of the shortest paths is:
\item[]
$102302\rightarrow 002302 \rightarrow 112302 \rightarrow 000002 \rightarrow 000003 \rightarrow 111113 \rightarrow 000013 \rightarrow 111013\rightarrow 001013\rightarrow 101013$.

\end{itemize}

%


\subsection*{Acknowledgment}
This research was supported by the
{\em Ministerio de Ciencia e Innovaci\'on} (Spain) and the {\em European Regional
Development Fund} under project MTM2011-28800-C02-01, and the {\em Catalan Research
Council} under project 2014SGR1147.

\bibliographystyle{plain}

\begin{thebibliography}{99}


%

\bibitem{ajb99}
R. Albert, H. Jeong, and A.-L. Barab\'asi,
Diameter of the world wide web,
{\em Nature} {\bf 401}  (1999), 130--131.

\bibitem{ba99}
A.-L. Barab\'asi and R. Albert,
Emergence of scaling in random networks,
{\em Science} {\bf 286} (1999), 509--512.

\bibitem{bo04}
A.-L. Barab\'asi and Z.N. Oltvai,
Network biology: Understanding the cell's functional organization,
{\em Nature Rev. Genetics} {\bf 5} (2004), 101--113.

%

\bibitem{bcdf09}
L. Barri\`ere, F. Comellas, C. Dalf\'o and M.A. Fiol, The
hierarchical product of graphs, {\em Discrete Appl. Math.} {\bf 157}
(2009), no. 7, 36--48.

\bibitem{bccf09}
L. Barri\`ere, F. Comellas, C. Dalf\'o, and M.A. Fiol, On the
hierarchical product of graphs and the generalized binomial tree,
{\em Linear Multilinear Algebra} {\bf 57} (2009), no. 7, 695--712.

\bibitem{bcdf14}
L. Barri\`ere, F. Comellas, C. Dalf\'o, and M.A. Fiol,
Deterministic hierarchical networks, submitted, 2014.

%
%
%
%
%


%
%

%

\bibitem{gfy92}
J. G\'omez, M.A. Fiol, and J.L.A. Yebra, Graphs on alphabets as models for large interconnection networks, {\em Discrete Appl. Math.} {\bf 37/38} (1992) 227--243.

%
\bibitem{JeMaBaOl01}
H. Jeong, S. Mason, A.-L. Barab\'asi, and Z.N. Oltvai,
Lethality and centrality in protein networks,
Nature 411 (2001) 41--42.

%

\bibitem{tp12}
My T. Thai and P.M. Pardalos (eds.),
{\em Handbook of Optimization in Complex Networks}, Springer Optimization and Its Applications {\bf 57}, Springer-Verlag, New Yoek, 2012.

%
%
%

\bibitem{n03}
J.D. Noh,
Exact scaling properties of a hierarchical network model,
{\em Phys. Rev. E} {\bf 67}  (2003), 045103.

\bibitem{rb03}
E. Ravasz and A.-L. Barab\'asi,
Hierarchical organization in complex networks,
{\em Phys. Rev. E} {\bf 67} (2003), 026112.

\bibitem{rsmob02}
E. Ravasz, A. L. Somera, D. A. Mongru, Z. N. Oltvai, and A.-L. Barab\'asi,
Hierarchical organization of modularity in metabolic networks,
{\em Science} {\bf 297} (2002), 1551--1555.

%
%
%
%
%
%

\bibitem{ws98}
D.J. Watts and S.H. Strogatz,
Collective dynamics of `small-world' networks,
{\em Nature} {\bf 393} (1998), 440--442.

\bibitem{wrb03}
S. Wuchty, E. Ravasz, and A.-L. Barab\'asi,
``The Architecture of Biological Networks'',
Complex Systems in Biomedicine, T.S. Deisboeck, J. Yasha Kresh and T.B. Kepler (Editors), Kluwer Academic Publishing, New York, 2003.


\end{thebibliography}

\end{document}